\documentclass[11pt,notitlepage,twoside,a4paper]{amsart}
 \usepackage{amsfonts}

\usepackage{amsmath,amssymb,enumerate}

\usepackage{epsfig,fancyhdr,color}

\usepackage{amssymb}
\usepackage{amsmath,amsthm}
\usepackage{latexsym}
\usepackage{amscd}
\usepackage{psfrag}
\usepackage{graphicx}
\usepackage[latin1]{inputenc}
\usepackage[all]{xy}
\usepackage[mathcal]{eucal}

\definecolor{NoteColor}{rgb}{1,0,0}

\renewcommand{\textsc}{\textcolor{red}}

\newtheorem{theorem}{\rm\bf Theorem}[section]
\newtheorem{proposition}[theorem]{\rm\bf Proposition}
\newtheorem{lemma}[theorem]{\rm\bf Lemma}

\newtheorem*{theorem 1}{\rm\bf Proposition 1}
\newtheorem*{theorem 2}{\rm\bf Proposition 2}

\theoremstyle{definition}
\newtheorem{definition}[theorem]{\rm\bf Definition}

\theoremstyle{remark}

\newtheorem*{problem}{\rm\bf Problem}

\def\interieur#1{\mathord{\mathop{\kern 0pt #1}\limits^\circ}}

\title[Transitional geometry]{transitional geometry}

\author{Norbert A'Campo and Athanase Papadopoulos}
 \address{Universit\"at Basel,  Mathematisches Institut, 
\\
Rheinsprung 21, 4051 Basel, Switzerland
\\
and 
Erwin Schr\"odinger International Institute of Mathematical Physics
\\
Boltzmanngasse 9, 1090, Wien, Austria
}
\email{norbert.acampo@gmail.com} 
\address{Athanase Papadopoulos,  Institut de Recherche Math\'ematique Avanc\'ee,
Universit\'e de Strasbourg and CNRS,
7 rue Ren\'e Descartes,
 67084 Strasbourg Cedex, France
 \\
Erwin Schr\"odinger International Institute of Mathematical Physics
\\
Boltzmanngasse 9, 1090, Wien, Austria
\\
and Galatasaray University,
Ciragan Cad. No:36, Ortak\"oy, 34349 Istanbul
} 
 \email{athanase.papadopoulos@math.unistra.fr}

\date{\today}

\begin{document}

\begin{abstract}
We develop a transitional geometry, that is, a family of geometries of constant curvatures which makes a continuous connection between the hyperbolic,  Euclidean and spherical geometries.
 In this transitional setting,  several geometric entities like points, lines, distances, triangles, angles, area, curvature, etc. as well as trigonometric formulae and other properties transit in a continuous manner from one geometry to another. 
 
 \medskip

\bigskip 


\bigskip 


 \bigskip

 \noindent The final version of this paper will appear in the volume \emph{Lie and Klein: The Erlangen program and its impact in mathematics and physics}, ed. L. Ji and A. Papadopoulos,  Z\"urich,  European Mathematical Society (EMS).

\noindent  The authors were supported by  the French ANR project FINSLER, by the GEAR network of the National Science Foundation (GEometric structures And Representation varieties) and by the Erwin Schr\"odinger Institute (Vienna). The second author was also supported by a TUBITAK 2221 grant for a visit to Galatasaray University (Istanbul).

\bigskip

 \end{abstract}
\maketitle

\tableofcontents

 \section{Introduction} \label{s:tr}
 
In this paper, we relate the three geometries of constant curvature by constructing a continuous transition between them.  We construct a fiber space $\mathcal{E}\to [-1,1]$, where the fiber above each point $t\in [-1,1]$ is a two-dimensional space of constant curvature which is positive for $t>0$ and negative for $t<0$ and such that when $t$ tends to $0$ from  each side, the geometry converges to the geometry of the Euclidean plane. Making the last statement precise will consist in showing that segments, angles, figures and several properties and propositions of the two non-Euclidean geometries converge in some appropriate sense to those of the Euclidean one. We say that the geometries \emph{transit}. At the level of the basic notions (points, lines, distances, angles, etc.) and of the geometrical properties of the figures (trigonometric formulae, area, etc.), we say that they  \emph{transit in a coherent way}.\index{coherent}\index{transition of geometries} This is expressed by the existence of some analytic sections of the fiber space $\mathcal{E}$ which allows us to follow these notions and properties. The group-theoretic definitions of each of the primary notions of the geometries of constant curvature (points, lines, etc.) make them transit in a coherent way.

 The idea of transition of geometries was already emitted by Klein in \S 15 of his paper \cite{Klein-Ueber} (see also the comments in   \cite{ACPK1}). We note however that the abstract theoretical setting of transformation groups was not really developed at the time Klein wrote the paper \cite{Klein-Ueber}. The notion of ``classical group", which is probably the most convenient setting for the description of these groups, was developed later, in works of Weyl \cite{Weyl-classical}, Dieudonn\'e \cite{Dieudonne-classical} and others. 
 
 The construction we describe here is also developed in Chapter 9 of our \emph{Notes on hyperbolic geometry} \cite{ACP}. We review the basic ideas and we add to them some new results and comments. Beyond the fact that it establishes relations between the three classical geometries, this theory highlights at the same time the important notions of \emph{families} and of \emph{deformations}.
  
  \section{The fiber space $\mathcal{E}$}
   
We work in a fiber space over $[-1,1]$ where the fibers above each point are built out of the groups of transformations of the three geometries. Each fiber is a geometry of constant curvature considered as a homogeneous space in the sense of Lie group  theory, that is, a space of cosets $G/H_0$ where $G$ is a Lie group -- the orthogonal group of some quadratic form -- and $H_0$ the stabilizer of a point. This is an example of the classical concept of \emph{Klein geometry},\index{Klein geometry} see e.g. \cite{Frances} in this volume.

When the parameter $t\in [-1,1]$ is negative, the homogeneous space is the hyperbolic plane of a certain constant negative curvature. When it is positive, the homogeneous space is the sphere (or the elliptic plane) of a certain constant positive curvature. When the parameter is zero, the homogeneous space is the Euclidean plane. We can define points, lines and other notions in the geometry using the underlying group. For instance, a \emph{point}\index{point} will be a maximal abelian compact subgroup of $G$ (which can be thought of as the stabilizer group of the point in the ambient group). The \emph{space} of the geometry is the set of points, and it is thus built out of the group. A \emph{line} is a maximal set of maximal abelian compact subgroups having the property that the subgroup generated by any two elements in this set is abelian. Other basic elements of the geometry (angle, etc.) can be defined in a similar fashion. This is in the spirit of Klein's Erlangen program where a geometry consists of a group action. In this case, the underlying space of the geometry is even constructed from the abstract group. The definitions are made in such a way that points, distances, angles, geometrical figures and trigonometric formulae vary continuously from hyperbolic to spherical geometry, transiting through the Euclidean, and several phenomena can be explained in a coherent manner.

To be more precise, we consider the vector space $\mathbb{R}^3$ equipped with a basis, and we denote the coordinates of a point $p$ by $(x(p),y(p),z(p))$ or more simply $(x,y,z)$.

We recall that the isometry groups of the hyperbolic plane and of the sphere are respectively the orthogonal groups of the quadratic form
\[
(x,y,z)\mapsto -x^2-y^2+z^2
\]
   and 
 \[(x,y,z)\mapsto  x^2+y^2+z^2.\]  
 Introducing a nonzero real parameter $t$ does not make a difference at the level of the axioms of the geometries (although it affects the curvature): for any $t<0$ (respectively $t>0$) the isometry group of the hyperbolic plane (respectively the sphere) is isomorphic to the orthogonal group of the quadratic form  
\begin{equation}\label{qf}
q_t:(x,y,z)\mapsto tx^2+ty^2+z^2.
\end{equation} 

We wish to include the Euclidean plane in this picture.

The first guess to reach the Euclidean plane is to give the parameter $t$ the value 0. This does not lead  to the desired result. In fact, although the orthogonal groups of the quadratic forms $q_t$, for $t>0$ (respectively for $t<0$) are all isomorphic, they are not uniformly bounded in terms of $t$, and when $t\to 0$, their dimension blows up.
 Indeed, when $t\to 0$ from either side, the quadratic form $q_t$ reduces to  
  \[(x,y,z)\mapsto z^2\] whose orthogonal group is much larger than the isometry group of the Euclidean plane. 
      
  Thus, an adjustment is needed. For this, we shall introduce the notion of ``coherent element". This will make the Euclidean plane automorphism group (and the Euclidean plane itself) appear in a continuous way between the hyperbolic and spherical automorphism group (respectively the hyperbolic plane and the sphere).  
  
We now define the fiber space $\mathcal{E}$, equipped with its fibration
 \[\mathcal{E}\to [-1,1].
 \]
The fiber $E_t$ above each $t\in [-1,1]$ will be a space of constant positive curvature $t^2$ for $t>0$ and of constant negative curvature $-t^2$ for $t<0$, such that when $t$ converges to $0$ from either side, $E_t$ converges to the Euclidean plane. 
     
 We explain this now. It will be useful to consider the elements of the space $E_t$ as matrices, and when we do so we shall denote such an element by a capital letter $A$.

Let $J\subset\mathrm{GL}(3,\mathbb{R})\times [-1,1]\to [-1,1]$ be the fibration whose fiber $J_t$ above $t$ is
the subgroup of $\mathrm{GL}(3,\mathbb{R})$ consisting of the stabilizers of the form $q_t$. Finally, let $I\subset J$
be defined by taking fiberwise the connected component of the identity element in the orthogonal group of the quadratic form $q_t$, for each $t\in [-1,1]$. 
 
In other words, for every $t\not=0$, $I_t$ is the orientation-preserving component of the orthogonal group of the quadratic form $q_t$.  It is a Lie group of dimension 3. However, as we already noticed, $I_0$ is a Lie group  of dimension 6.  It is the matrix group
\[I_0= \{A=(A_{ij})\in \mathrm{GL}(3,\mathbb{R}) \ \vert \ A_{3,1}=A_{3,2}=0, A_{3,3}=1, \mathrm{det}(A)>0\}.\]
This is the group of matrices of the form
\[\left( \begin{array}{ccc}
a & b & e \\
c & d & f \\
0 & 0 & 1 
\end{array} \right)\] 
with $ad-bc>0$. 
The fact that the determinant is positive is a consequence of the orientation-preserving assumption. Thus, $I_0$ is the usual matrix representation group of the orientation-preserving group of affine transformations of $\mathbb{R}^2$. Equivalently, it is the group of matrices preserving the $(x,y)$-plane in $\mathbb{R}^3$. (One can deduce this fom the fact that the affine group is the subgroup of the group of projective transformations that preserve a hyperplane in projective space.) This shows again that $I_0$ is not isomorphic to the group of orientation-preserving Euclidean motions of the plane. We shall reduce the size of $I_0$, by restricting the type of matrices that we consider.

\begin{definition}[Coherent element]
An element $A$ of $I_0$ is said to be \emph{coherent} if for all $i$ and $j$ satisfying $1\leq i,j\leq 3$, there exists an analytic function $A_{i,j}(t)$ such that for each $t\in [-1,1]$ the matrix $A_t=\left(A_{i,j}(t)\right)$ belongs to $I_t$, and $A_0=A$. 
\end{definition}

\section{The space of coherent elements}
We denote by $E\subset I \subset   \mathrm{GL}(3,\mathbb{R}) \times  [-1,1] $ the set of coherent elements. For each $t\in [-1,1]$ we denote by $E_t$ the set of coherent elements that are above the point $t$.

From the definition, we have $E_t=I_t$ for all $ t\not=0$. We now study $E_0$. 

 The coherent elements of $I_0$ form a group and they have important features. We start with the following:
\begin{proposition}
If $A\in I_0$ is coherent, then $\mathrm{det}(A)=1$.
\end{proposition}
\begin{proof} For $t\not=0$, we have $\mathrm{det}(A)=1$ since $A$ is an orthogonal matrix of a non-degenerate quadratic form. The result then follows from the definition of coherence and from the fact that the map $t\mapsto \mathrm{det}(A_t)$ is continuous. 
\end{proof}

We also have the following:

\begin{proposition} \label{pro:coh}
An element $A\in I_0$ is coherent if and only if $A$ is an orientation-preserving motion of the Euclidean plane.
\end{proposition}

\begin{proof}

Let  $A\in I_0$ be a coherent element. From the definition, we have $A=\lim_{t\to0} A(t)$, with $A(t)\in I_t$, $ t>0$. Thus, the 9 sequences $A_{ij}(t)$, for $1\leq i,j\leq 3$, converge and therefore they are bounded. For $t\not=0$, since $A(t)$ preserves the quadratic form $q_t$, therefore it also preserves the form $\frac{1}{t}q_t$.  For $i=1,2,3$, let  $A_i(t)$ denote the $i$-th column of $A(t)$. We show that for $i,j<3$, the  $\frac{1}{t}q_t$-scalar product of  $A_i(t),A_j(t)$ is the $i,j$-th coefficient of the Kronecker delta function.

We have 
\[\frac{1}{t}q_t(A)= 
\left( \begin{array}{ccc}
A_{11}(t) & A_{21}(t) & A_{31} (t)
\end{array} \right)
\left( \begin{array}{ccc}
1 & 0 &0  \\
0 & 1 &0  \\
0 &0 & \frac{1}{t} \\
\end{array} \right)
\left( \begin{array}{ccc}
A_{11} (t)&   \\
A_{21}(t) &   \\
A_{31} (t)& \\
\end{array} \right)
\]
\[
= A_{11}(t)^2+A_{21}(t)^2+\frac{1}{t} A_{31}(t)^2.
\]

Since
$A(t)$ preserve $\frac{1}{t}q_t$, we get $A_{11}(t)^2+A_{21}(t)^2+\frac{1}{t} A_{31}(t)^2=1$ for all $t\not=0$. 
Since the sequences $ A_{ij}(t)$ converge, we obtain $\lim_{t\to 0} A_{31}(t)=0$, therefore the coherence property imposes $A_{31}(0)=0$. Now we use the fact that the section $A_{31}(t)$ is differentiable (recall that by coherence, it is even analytic). This gives $A_{31}(t)=0+\lambda t+O(t^2)$, therefore $ \frac{1}{t} A_{31}(t)^2=\lambda^2 t+O(t^3)$, which implies  $\lim_{t\to 0}\frac{1}{t} A_{31}(t)^2=0$. We get $A_{11}(0)^2+A_{21}(0)^2=1$.

In the same way, we prove that $A_{12}(0)^2+A_{22}(0)^2=1$ and $A_{11}(0)A_{12}(0)+A_{21}(0)A_{22}(0)=0$. 

It follows that the first principal 2-bloc of 
$A$ is an orthogonal matrix, which implies that $E_0$ is a direct isometry group of the Euclidean plane. In other words,  the $2\times
 2$ matrix 
\[\left( \begin{array}{cc}
A_{11} & A_{12}  \\
A_{21} & A_{22}  \\
\end{array} \right)\] 
preserves the standard quadratic form on $\mathbb{R}^2$. Thus, it is an element of the Euclidean rotation group $\mathrm{SO}(2)$.
\end{proof}

\noindent \emph{Note.---} In the proof of the preceding proposition, we used the fact that the coherence property imposes that the sections $A_{ij}$ are differentiable. There is another proof in \cite{ACP} (Proposition 2.2) which only uses the existence of  sections that are only continuous.
  
  \medskip
  
 The set of coherent elements of $I_0$, being the group of orientation-preserving motions of the Euclidean plane, is the group we are looking for.

From now on, we use the notation $I_0$ for the subgroup of \emph{coherent}\index{coherent} elements instead of the previously defined group $I_0$. Thus, in the following, for each $t\in [-1,1]$, any element of $I_t$ is coherent.

The space of coherent elements (for variable $t$) is not a locally trivial fiber bundle (for instance, the topological type of the fiber $I_t$ is not constant; it is compact for $t>0$ and noncompact for $t<0$). However this space has nice properties which follow from the fact that it has many continuous (and even analytic) sections.  

\section{The algebraic description of points and lines} 
  
   For every $t\in [-1,1]$, we define $E_t$ as the set of maximal abelian compact subgroups of $I_t$.  An element of $E_t$ is a \emph{point} of our geometry. Note that any maximal abelian compact subgroups of $I_t$ is isomorphic to the circle group $\mathrm{SO}(2)$.  With this in hand, a coherent family of points in the fiber space $(E_t)$ is a coherent family of maximal abelian compact subgroups of $I_t$. This is a family depending analytically on the parameter $t$. We can also consider coherent families of pairs (respectively triples, etc.) of points. This defines the segments (respectively triangles, etc.) of our geometries. We can study the corresponding distance (respectively area, etc.) function of $t$ defined by such a pair or triple, in algebraic terms. 
   
   For each $p\in E_t$, we denote by $K_p\subset I_t$ the maximal subgroup that defines $p$.  Since $K_p$ is a group isomorphic to the circle group $\mathrm{SO}(2)$, for each $p\in E_t$, there exists a unique order-two element $s_p\in K_p$. We shall make use of this element. In the circle group $\mathrm{SO}(2)$, this corresponds to the rotation of angle $\pi$.  We call $s_p$ the \emph{reflection}, or \emph{involution} in $I_t$, of center $p$. 
In this way, any point in $E_t$ is represented by an involution. For each $t\not=0$, an involution is a self-map of the space of our geometry (the sphere or the hyperbolic plane) that fixes the given point, whose square is the identity, and whose differential at the given point is $-\mathrm{Id}$.
    
This algebraic description of points in $E_t$  as involutions has certain advantages. In particular, we can define compositions of involutions and we can use this operation to describe algebraically lines and other geometric objects in $E_t$.

A \emph{line}\index{line} in $E_t$ is a maximal subset $L$ of $E_t$ satisfying the following property:
\begin{quote}
(*) \hskip .1in 
 The subgroup of $I_t$ generated by the all elements of the form $s_p s_{p'}$, for $p, p'\in L$, is abelian. In the the case $t=0$, we ask furthermore that the group is proper. (For $t\not=0$, this is automatic).
\end{quote}

In other words, each time we take four points $s_{p_{1}}, s_{p_{2}}, s_{p_{3}}, s_{p_{4}}$ in $E_t$  represented by involutions, then, $s_{p_{1}}s_{p_{2}}$ commutes with $s_{p_{3}}s_{p_{4}}$.

Given two distinct points in $E_t$, there is a unique line joining them; this is the maximal subset $L$ of $E_t$ containing them and satisfying property (*) above.

 The group $I_t$ acts by conjugation on $E_t$ and by reflections along lines.
 
With these notions of points and lines defined group-theoretically, one can check that the postulates of the geometry can be expressed in group theoretic terms: any two distinct points belong to a line, two lines intersect in at most one point. (For this to hold, in the case $t>0$, one has to be in the projective plane and not on the sphere.)
 
We now introduce circles. Let $s$ be a point represented by a subgroup $K_p$ and an involution $s_p$. Any element $\rho$ in $K_p$ acts on the space of points $E_t$ by conjugation:
\[(\rho, K_p)\mapsto \rho K_p\rho^{-1}.\] A \emph{circle}\index{circle} of center $p$ in our geometry is an orbit of such an action. For $t>0$ and for every point $p$, among such circles, there is one and only one circle centered at $p$ which is a line of our geometry.

 A \emph{triangle}\index{triangle} in our space is determined by three distinct points with three lines joining them pairwise together with the choice of a connected component of the complement of these lines which contains the three given points on its boundary. The last condition is necessary in the case of the sphere, snce, in general, three lines divide the sphere into eight connected components.

Let us consider in more detail the case $t>0$. This case is particularly interesting, because we can formulate the elements of spherical polarity\index{polarity (sphere)} theory in this setting.  
We work in the projective plane (elliptic space) rather than the sphere. In this way,  the set of points of the geometry $E_t$  ($t>0$) can be thought of as the set of unordered pairs of antipodal points on the sphere in the $3$-dimensional Euclidean space $(\mathbb{R}^3, q_t)$. 

To each line in $E_t$ is associated a well-defined point called its \emph{pole}.\index{pole}  In algebraic terms (that is, in our description of points as involutions), the pole of a line is the unique involution $s_p$ which, as an element of $I_t$, fixes globally the line and does not belong to it.

Conversely, to each point $p$, we can associate the line of which $p$ is the pole. This line is called the \emph{equator}\index{equator} of $p$. There are several equivalent algebraic characterizations of that line. For instance, it is the unique line $L$ such that for any point $q$ on $L$, the involution $s_q$  fixes the point $p$.
In other words, the equator of a point $s_p$ is the set of points $q\not= p$ satisfying  $s_q(p)=p$.\footnote{This correspondence between poles and lines holds because we work in the elliptic plane, and not on the sphere. In the latter case, there would be two ``poles", which are exchanged by the involution $s_q$ associated to a point $q$ on the line.}

This correspondence between points and lines is at the basis of duality theory.

If two points $p$ and $q$ in $E_t$ are distinct,  the product of the corresponding involutions $s_p$ and $s_q$ is a translation along the line joining these points. More concretely, $s_p s_q$ is  a rotation along the line in 3-space which is perpendicular to the plane of the great circle determined by $p$ and $q$. This line passes through the pole $s_N$ of the great circle, and therefore the product  $s_p s_q$ commutes with the pole $s_N$ (seen as an involution).

Given $p\in E_t$, the equator of $p$ is the set of all $q\in E_t$ whose image under the stabilizer of $p$ in $I_t$ is a straight line.

In spherical (or elliptic)  geometry, a symmetry with respect to a point is also a symmetry with respect to a line.  This can be seen using the above description of points as involutions. From the definition, the involution corresponding to the pole of a line fixes the pole, but it also fixes pointwise the equator. The pole can be characterized in this setting as being the unique \emph{isolated} fixed point of its involution. The set of other fixed points is the equator. Thus, to a line in elliptic geometry is naturally associated an involution.

 We can use polarity to define perpendicularity between lines: Consider two lines $L_1$ and $L_2$ that intersect at a point $p$. Then $L_ 1$ and $L_2$ are \emph{perpendicular}\index{perpendicular} if the pole of $L_1$ belongs to $L_2$. This is equivalent to the fact that the polar dual to $L_2$ belongs to $L_1$. This also defines the notion of \emph{right angle}.\index{right angle}

In what follows, we shall measure lengths in the geometry $E_t$ for $t>0$. We know that in spherical (or elliptic) geometry, there is a natural length unit.\index{length unit (sphere)} The length unit in this space can be taken to be the diameter of a line (all lines in that geometry are homeomorphic to a circle), or as the diameter of the whole space (which is compact). We can also use the correspondence between lines and poles and define a normalized distance on $E_t$, by fixing once and for all the distance from a point to its equator.  
We normalize this distance by setting it equal to  $\displaystyle \frac{\pi}{2\sqrt{t}}$. (Remember that each space $E_t$, for $t>0$, is the sphere of constant curvature $t^2$.)
 
 After defining the distance on $E_t$, we can check that the segment which joins a pole to a point on its equator is the shortest path to the equator.

\section{Transition of points, lines, distances, curvature and triangles} \label{s:trans}
Before we continue, let us summarize what we did. We chose of a fixed basis for $\mathbb{R}^3$. We defined the space $\mathcal{E}=(E_t)_{t\in [-1,1]}$  of coherent elements as a subset of a space of matrices, equipped with a map onto $[-1,1]$. A point $p=(x,y,z)$ in $\mathbb{R}^3$ equipped with a basis gives a group-theoretically defined point in $E_t$ for every $t\in [-1,1]$ with $t\not=0$, namely, the stabilizer of the vector whose coordinates are $(x,y,z)$ in the group of linear transformations of $\mathbb{R}^3$ preserving the form $q_t$. Using the coherence property, we extended to $t=0$, and for any $t\in [-1,1]$, we defined a point $A_t$ in $E_t$ as a stabilizer group under the action of the group $I_t$, namely, a maximal abelian subgroup $K_t$ of $I_t$. The family $A_t$, $t\in [-1,1]$ is a coherent family of points. 
This group-theoretic definition of points allows the continuous transition of points in the space $\mathcal{E}$. Likewise, the group-theoretic definition of lines makes this notion coherent transit from one geometry to another. 

Now we introduce the notion of angle. We want a notion that transits between the various geometries. This can also  be done group-theoretically in a coherent manner. There are various ways of doing it, and one way is the following.  We start with the  group-theoretic description of a point as a maximal compact abelian subgroup (the ``stabilizer" of the point). Such a subgroup is isomorphic to a circle group. But the circle group is also the set of all oriented lines starting at the given point. It is equipped with its natural Haar measure, which we normalize so that the total measure of the circle is $2\pi$. This gives a measure on the set of oriented lines. From these measures, we obtain the notion of angle at every point. It is also possible to give a coherent  definition of an angle by using a formula, and we shall do this in \S \ref{s:angle}.  

Now we consider transition of distances.

To measure distances between points, we use a model for the space $E_t$ for each $t\in [0,1]$. We take the unit sphere $S_t$, that is, the space of vectors in $\mathbb{R}^3$ of norm one with respect to the quadratic form $q_t$. This is the subset defined by 
\begin{equation}\label{eq:ellipse}
S_t=\{(x,y,z)\ \vert \ tx^2+ty^2+z^2=1\} \subset \mathbb{R}^3.
\end{equation}
We take the quotient of this set by the action of $\mathbb{Z}_2$ that sends a vector to its opposite. Note that the unit sphere $S_t$ (before taking the quotient) is connected for $t>0$, and it has two connected components for $t<0$. After taking the quotient, all spaces become connected. The group of transformations of $\mathbb{R}^3$ which preserve the quadratic form $q_t$ acts transitively on $S_t/\mathbb{Z}_2$ for each $t$. For $t=0$, the unit sphere (before taking the quotient) is defined by the equation $z^2=1$, and it has two connected components, namely, the planes $z=1$ and $z=-1$.

In this way, the elements of the geometries $E_t$ can be seen either algebraically as cosets in a group (the isometry group modulo the stabilizer of a point) or, geometrically, as elements of the unit sphere quotiented by  $\mathbb{Z}_2$. 
 
 Now we define distances between pairs of points in each $E_t$ and then we study transition of distances.

 For $t\not= 0$, we use the angular distance in $E_t$. To do so, for each $t\not=0$, we let $\beta_t$ be the bilinear form associated to the quadratic form $q_t$, that is,
\[\beta_t(x,y)= \frac{q_t(x+y)-q_t(x)-q_t(y)}{2}.\]
We set, for every $x_t$ and $y_t$ in $E_t$,
\begin{equation}\label{eq:dis}
w_t(x_t,y_t)= \arccos \frac{\beta_t(x_t,y_t)}{\sqrt{q_t(x_t)}\sqrt{q_t(y_t)}}
\end{equation}
We call $w_t(x_t,y_t)$ the \emph{angular distance} between the points $x_t$ and $y_t$. For each $t\not=0$, the angular distance can be thought of as being defined on $S_t/\mathbb{Z}_2$, but also on the 2-dimensional projective plane, the quotient of $\mathbb{R}^3\setminus\{0\}/\mathbb{Z}_2$ by the nonzero homotheties.
As a distance on the space $E_t$, we shall take $w_t$ multiplied by a constant $c_t$ that we determine below. 
Note that although Equation (\ref{eq:ellipse}) describes an ellipsoid in the Euclidean orthonormal coordinates $(x,y,z)$, this surface, equipped with this angular metric induced from the quadratic form, is isometric to a round sphere. 
 
 For each $t\not=0$, equipped with the normalized distance function $c_tw_t$, the point set $E_t$ becomes a metric space. Its points and its geodesics coincide with the group-theoretically defined points and lines that we  considered in \S\ref{s:tr}. This can be deduced from the fact that the isometry group of the space acts transitively on points and on directions, and that both notions (the metric and the group-theoretic) are invariant by this action. There are other ways of seeing this fact. 
%
 
 We can draw a picture of $S_t$ for each $t\in [0,1]$. In dimension 2, this is represented in Figure \ref{fig:1}. To get the 3-dimensional picture, this figure has to be rotated in space around the $y$-axis.

 \begin{figure}[!hbp]
\centering
 \psfrag{O}{\small $O$}
\psfrag{x}{\small $x$}
\psfrag{y}{\small $y$}
\psfrag{P}{\small $t>0$}
\psfrag{N}{\small $t<0$}
\psfrag{Z}{\small $t=0$}
\includegraphics[width=0.90\linewidth]{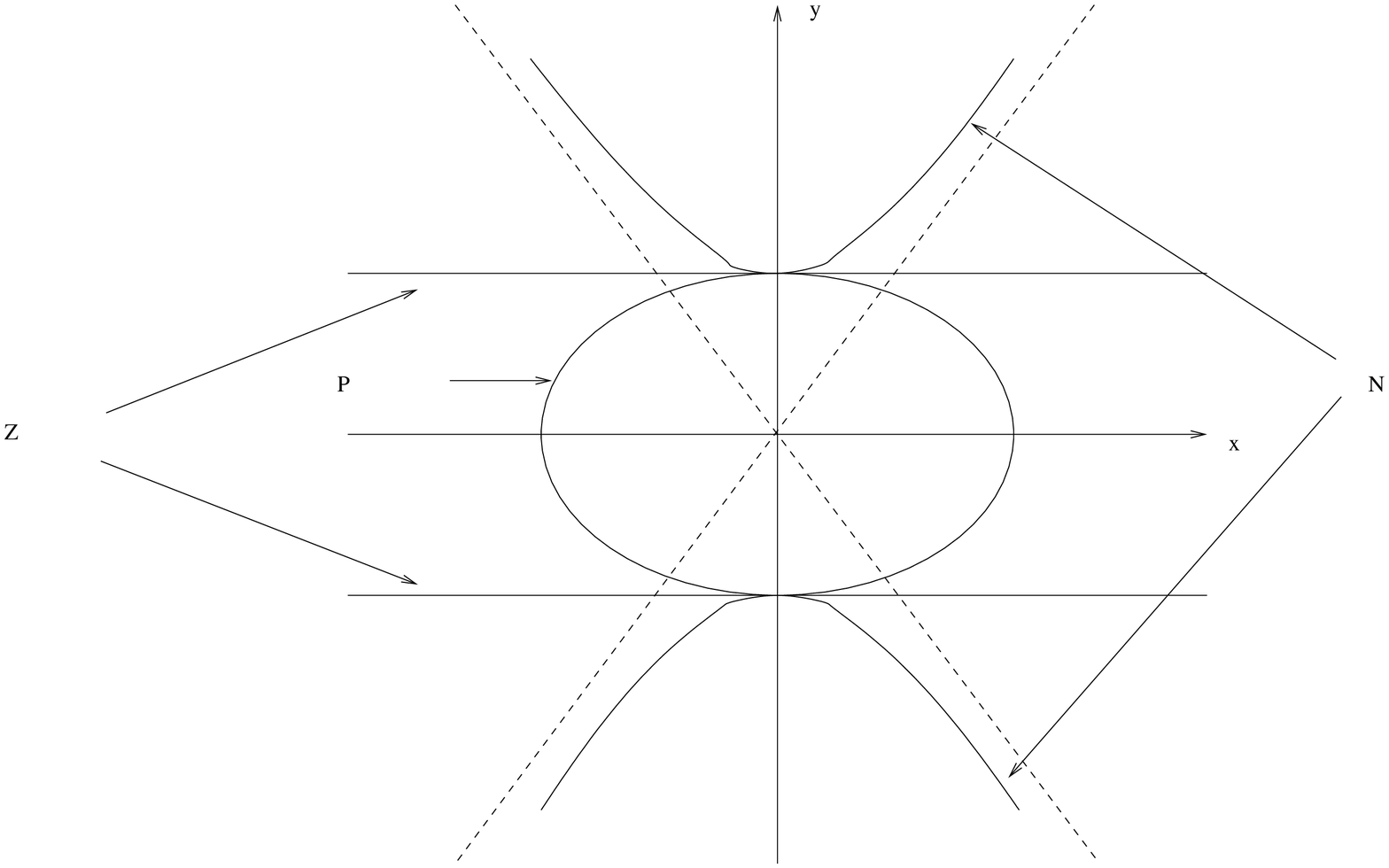}
\caption{\small }
\label{fig:1}
\end{figure}

Any vector in $\mathbb{R}^3$ defines a point in each geometry $E_t$ by taking the intersection of the ray containing it with $S_t$ (or in the quotient of this intersection by $\mathbb{Z}_2$). This is represented in Figure \ref{fig:2}, in which the vector is denoted by $P$, and where we see three intersection points with the level surfaces $S_t$: for $t>0$, $t=0$ and $t<0$. Note that the two points with Cartesian coordinates $(0,0,1)$ and $(0,0,-1)$ belong to all geometries (they belong to $S_t$ for any $t\in [-1,1]$). We shall use this in the following discussion about triangles, where some vertices will be taken to be at these points, and this will simplify the computations.

 \begin{figure}[!hbp]
\centering
 \psfrag{O}{\small $O$}
\psfrag{x}{\small $x$}
\psfrag{y}{\small $y$}
\psfrag{P}{\small $t>0$}
\psfrag{N}{\small $t<0$}
\psfrag{Z}{\small $t=0$}
\includegraphics[width=0.60\linewidth]{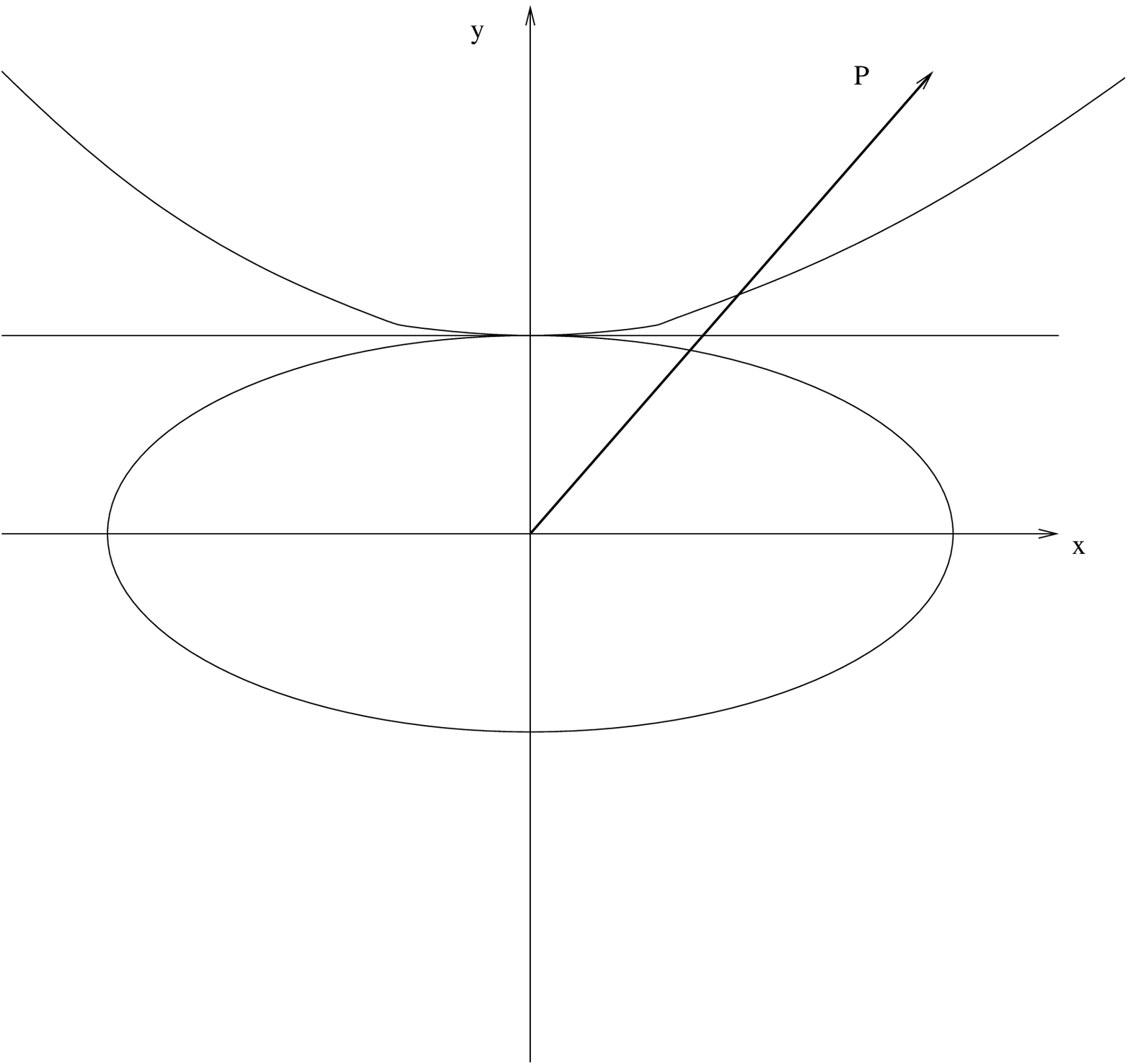}
\caption{\small }
\label{fig:2}
\end{figure}

  We now discuss transitions of distances.

We want the distance between two points in the geometry $E_0$ to be the limit as $t\to 0$ (from both sides,  $t>0$ and $t<0$) of the distance in $E_t$ between corresponding points (after normalization). The next proposition tells us how to choose the normalization factor $c_t$.

Let $A$ and $B$ be two points in $\mathbb{R}^3$. We have a natural way of considering each of these points to lie in $E_t$, for every $t\in [-1,1]$. Let us denote by $A_t,B_t$ the corresponding points.
\begin{proposition} \label{p:normalization}
 The limit as $t\to 0, \ t>0$ (respectively $t<0$), of the distance from $A_t$ to $B_t$ in $E_t$ normalized by the factor $1/\sqrt{t}$ is equal to the Euclidean distance between $A_0$ and $B_0$. 
\end{proposition}

Using this result, we shall define the distance between points $A$ and $B$ in $E_0$ in such a way that the distances in Proposition \ref{p:normalization} vary continuously for $t\in [-1,1]$.

 \begin{proof}
For $t\in [-1,1]$, let $a=(0,0,1)\in \mathbb{R}^3$ and consider the
coherent family of points $A_t\in E_t, t\in [-1,1]$ represented by $a$. For $x\in \mathbb{R}$, let $B_t^x\in E_t$, 
$t\in [-1,1]$ be the coherent family of points represented by the stabilizer of the vector $b^x=(x,0,1)$ of $\mathbb{R}^3$. The family 
$[A_t,B_t^x]$, $t\in [0,1]$ is a coherent family of segments in $\mathcal{E}$.

We start with $t>0$. We compute the limit as $t\to 0, \ t>0$, of the distance from $A_t$ to $B_t^x$ in $E_t$, normalized by dividing it by $\sqrt{t}$. We have \[q_t(a)=1,\] 
\[q_t(b^x)=tx^2+1\]
 and 
 \[\beta_t(a,b^x)=1.\]

We use the angular distance from $A_t$ to $B_t^x$, measured with $q_t$, $t>0$. By Equation (\ref{eq:dis}), we have
 \begin{eqnarray*}
w_t(A_t, \beta_t^x)&=&\arccos\ (\frac{\beta_t(a,b^x)}{\sqrt{q_t(a)}\sqrt{q_t(b^x)}})\\
&=&\arccos (\frac{1}{\sqrt{tx^2+1}})\\
&=&\vert \arctan (\sqrt{t}x)\vert\\
& =& \vert \sqrt{t}x-\frac{1}{3}\sqrt{t^3}x^3+\frac{1}{5}\sqrt{t^5}x^5-\ldots\vert
\end{eqnarray*}
Taking $c_t=1/\sqrt{t}$ as a factor for distances in $E_t$, the distance $c_tw_t$ between the two points becomes

\begin{equation}\label{11}
D_t(A_t,B_t^x)=\frac{1}{\sqrt{t}}\arccos (\frac{1}{\sqrt{tx^2+1}})=
\vert x-\frac{1}{3}tx^3+\frac{1}{5}t^2 x^5-\ldots \vert.
\end{equation}

The limit as $t\to 0$, $t>0$, of $D_t(A_t,B_t^x)$ is $\vert x\vert$, which is what the distance from $A_0$ to $B_0^x$ in $E_0$ ought to be.
\end{proof}
Thus, we define the distance $D_t$ on each $E_t$ to be the angular distance $w_t$ given in (\ref{eq:dis}) divided by $\sqrt{t}$. In this way, the distance function transits.

The fact that points and distances transit between the various geometries implies that triangles also transit. We already noted that angles also transit. From this, we get several properties. For instance, since angle bisectors may be defined using equidistance,  the following property transits between the various geometries: 

\begin{proposition}
The angle bisectors in a triangle intersect at a common point.
\end{proposition}

The reader can search for other properties that transit between the various geometries.

Curvature is also a coherent notion. Indeed, curvature in spherical and in hyperbolic geometry is defined as the integral of the excess (respectively the defect) to two right angles of the angle sum in a triangle (and to a multiple of a right angle for more general polygons). This notion transits through Euclidean geometry, which is characterized by zero excess to two right angles. One can make this more precise, by choosing a triangle with fixed side lengths -- for instance with  three sides equal to a quarter of a circle -- and making this triangle transit through the geometries.

\section{Transition of trigonometric formulae} 

The trigonometric formulae, in a geometry, make relations between lengths and angles in triangles. They are at the bases of the geometry, since from these formulae, one can recover the geometric properties of the space. Let us recall the formal similarities between the trigonometric formulae in hyperbolic and in spherical geometry. One example is the famous ``sine rule", which is stated as follows:

For any  triangle $ABC$, with sides  (or side lengths) $a,b,c$ opposite to the vertices $A,B,C$, we have, in Euclidean geometry :
\[\frac{a}{\sin A} =\frac{b}{\sin B}= \frac{c}{\sin C}, \] in spherical geometry:
\[\frac{\sin a}{\sin A} =\frac{\sin b}{\sin B}= \frac{\sin c}{\sin C} \] and in hyperbolic geometry:
\[\frac{\sinh a}{\sin A} =\frac{\sinh b}{\sin B}= \frac{\sinh c}{\sin C} .\] 

Thus, to pass from the Euclidean to the spherical (respectively the hyperbolic), one replaces a side length by the cosine of the side length (respectively the sine of the length). Since using the sine rule one can prove many other trigonometric formulae, it is natural to expect that there are many occurrences of trigonometric formulae in non-Euclidean geometry where one replaces the circular functions of the side lengths (sin, cos, etc.) by the hyperbolic function of these side lengths, in order to get the hyperbolic geometry formulae from the spherical ones. Examples are contained in the following table. They concern a triangle $ABC$ having a right angle at $C$:
 
 \begin{table}[h]
 \begin{center}
 \label{table1}
 \renewcommand{\arraystretch}{1.25}
 \begin{tabular}{  | l | l | l | r || }
 \hline
Hyperbolic & Euclidean & Spherical \\ \hline  \hline
  $\cosh c = \cosh a \cosh b$ & $c^2= a^2+b^2$ &   $\cos c = \cos a \cos b$ \\ \hline
  $\sinh b = \sinh c \sin B$ & $b=c\sin B$ &    $\sin b = \sin c \sin B $  \\ \hline
     $\tanh a = \tanh c \cos B$ & $a=c\cos B $ & $\tan a = \tan c \cos B $\\ \hline
     $\cosh c = \cot A \cot B $& $1=\cot A\cot B$ &    $\cos c = \cot A \cot B $  \\ \hline
     $\cos A = \cosh a \sin B $&$ \cos A=\sin B$ &  $\cos A = \cos a \sin B $ \\ \hline
   $\tan a = \sinh b \tan A $& $a=b\tan A$ &  $\tan a = \sin b \tan A $  \\ \hline
 \end{tabular}
 \medskip
 \bigskip
 \end{center}
\end{table}

In this table, the Euclidean formulae in the middle column are obtained by taking Taylor series expansions of any of the two corresponding non-Euclidean ones, for side lengths tending to $0$. This is a consequence of the fact that the spherical and the hyperbolic geometries become Euclidean at the level of infinitesimal triangles.
 
The transitional geometry sheds a new light on the analogies between the trigonometric formulae.

Fisrt, we must recall that the above formulae are valid for spaces of constant curvatures $-1$, $0$ and $1$. In a geometry $E_t$, with $t\not=0$, one has to introduce a parameter $t$ in the above trigonometric formulae. For instance, for a right triangle $ABC$ with right angle at $C$, we have, in spherical geometry ($t>0$) :
\[\cos  \sqrt{t}c = \cos \sqrt{t}a\cos \sqrt{t}b\] 
and in hyperbolic geometry ($t<0$) :
\[\cosh \sqrt{-t}a = \cosh \sqrt{-t}b \cosh \sqrt{-t}c .\]

Now we can study the transition of this formula, called the \emph{Pythagorean theorem}.\footnote{By extension from the Euclidean case, a Pythagorean theorem, in a certain geometry, is a formula that makes relations between the edges and angles of a right triangle.}

We first study the transition of a triangle.

For $t\in [-1,1]$, let $a=(0,0,1)\in \mathbb{R}^3$ and consider the
coherent family of points $A_t\in E_t, t\in [-1,1]$ represented by the vector $a$. Take a real number $x$. For any $t\in [0,1]$, let $B_t^x\in E_t$ and $C_t^x\in E_t$ be the two coherent families of points represented by the stabilizers in 
 $I_t$ of the vectors $b^x=(x,0,1)$ and $c^x= (0,x,1)$ of $\mathbb{R}^3$. The family 
$\Delta_t(x,y)=(A_t,B_t^x,C_t^y)$, $t\in [0,1]$ is a coherent family of triangles.
 The triangle $\Delta_0(x,y)$ is a right triangle at $A$, with catheti ratio $x/y$. 
To see that each triangle $\Delta_t(x,y)$ is right, we note that if we change $B$ into $-B$, we do not change the angle value. (Recall the definition of the angle measure as the Haar measure on $\mathrm{SO}(2)$.) Therefore there is a transformation of the space $E_t$ that belongs to our isometry group and that sends this angle to its opposite. Therefore the angle is equal to its symmetric image, and therefore it is right. 
We saw (proof of Proposition 2.1)  that in the geometry $E_t$, the distance between the two points $A_t$ and $B_t^x$ is given by
\begin{equation}\label{11}
D_t(A_t,B_t^x)=\frac{1}{\sqrt{t}}\arccos (\frac{1}{\sqrt{tx^2+1}})
\end{equation}
  Likewise, the distance between the two points $A_t$ and $C_t^y$ is equal to 
  \begin{equation}\label{15}
D_t(A_t,C_t^y)=\frac{1}{\sqrt{t}}\arccos (\frac{1}{\sqrt{ty^2+1}})
\end{equation}

In the geometry $E_0$, the distance from  from $A_0$ to $C_0^y$ is equal to $\vert y\vert$.

We need to know the distance in $E_0$ between the points $B^x_0$ and $C^y_0$.

The angular distance from $B_t^x$ to $C_t^y$, measured with $q_t$, $t>0$, is  
\begin{eqnarray*}
w_t(B_t^x, C_t^y)&=&\arccos\ (\frac{\beta_t(b^x, c^y)}{\sqrt{q_t(b^x),q_t(c^y)}})\\
&=&
\arccos (\frac{1}{\sqrt{tx^2+1}\sqrt{ty^2+1}})\\
&=&\arccos (\frac{1}{\sqrt{t^2x^4y^4+tx^2+ty^2+1}}).
\end{eqnarray*}

Up to higher order terms,  this expression is equal to

\[\sqrt{ t^2x^4y^4+tx^2+ty^2}.\] 
 
We collect the information in the following lemma, using the fact that in the coherent geometry $E_0$, distances are obtained as a limit of normalized distances in $E_t$ for $t>0$ (or $t<0$). 
\begin{lemma}  In the geometry $E_t$, we have 
 \[
D_t(A_t,B_t^x)=\frac{1}{\sqrt{t}}\arccos (\frac{1}{\sqrt{tx^2+1}}),
\]
\[
D_t(A_t,C_t^y)=\frac{1}{\sqrt{t}}\arccos (\frac{1}{\sqrt{ty^2+1}}),
\]
and
\[D_t(B_t^x,C_t^y)= \frac{1}{\sqrt{t}}  \arccos (\frac{1}{\sqrt{tx^2+1}\sqrt{ty^2+1}}).\]
In the coherent geometry $E_0$, we
 have $$\lim_{t\to 0} D_t(A_t,B_t^x)= D_0(A_0,B_0^x)= \vert x\vert,$$ 
  $$\lim_{t\to 0} D_t(A_t,C_t^x)=  D_0(A_0,C_0^y)= \vert y\vert$$ and 
$$\lim_{t\to 0} D_t(B_t,C_t^x)=  D_0(B_0^x,C_0^y)= \sqrt{x^2+y^2}.$$
\end{lemma}

Using this lemma, we now show that the Euclidean Pythagorean formula in $E_0$ is a limit of the Pythagorean formula in $E_t$ for $t>0$.

Let $x$ and $y$ be real numbers and let $t>0$. In the triangle $\Delta_t(x,y)$, the lengths of the three sides (measured in $E_t$) are the angular distances $\sqrt{t}D_t(A_t,B_t^x)$, $\sqrt{t}D_t(A_t,C_t^y)$ and $\sqrt{t}D_t(Cy_t,B_t^x)$. The first natural attempt is to write the Pythagorean theorem in spherical geometry, for the triangle with right angle at $A_t$:
 \begin{eqnarray*}
 \cos(\sqrt{t}D_t(A_t,B_t^x))
  \cos(\sqrt{t}D_t(A_t,C_t^y)) & = & \frac{1}{\sqrt{tx^2+1}\sqrt{ty^2+1}}\\
  &=& \cos(\sqrt{t}D_t(C^y_t,B_t^x)).
 \end{eqnarray*}

Taking the limit, as $t\to 0$ gives
\[1=1\times 1,\]
which is correct but which is not the Euclidean Pythagorean theorem.
We obtain a more useful result by taking another limit.

We transform the spherical Pythagorean theorem into the following one:
\[
\sqrt{t}\left(1-\cos(\sqrt{t}D_t(A_t,B_t^x))\cos(\sqrt{t}D_t(A_t,C_t^x))\right)=\]
\[
 \sqrt{t}\left(1-\cos(\sqrt{t}D_t(C^y_t,B_t^x))\right).\]
  Writing this equation at the first order using the expansion
 \[\frac{1}{\sqrt{t}}\arccos (\frac{1}{\sqrt{tx^2+1}})=\vert x-\frac{1}{3}tx^3+\frac{1}{5}t^2 x^5-\ldots \vert\] 
  and taking the limit as $t\to 0$, with $t>0$ (use Equation \ref{11}) gives
    \[D_0(A_0,B_0^x)^2 + D_0(A_0,C_0^y)^2=D_0(B_0^x,C_0^y)^2,\]
  which is the Pythagorean theorem in the geometry $E_0$. This is indeed the familiar Pythagorean theorem in Euclidean geometry.

 It is possible to do the same calculation in the hyperbolic case ($t<0$), where the hyperbolic Pythagorean theorem is:
\begin{proposition}
\begin{eqnarray*}
\cosh(\sqrt{-t}D_t(A_t,B_t^x)) \cosh(\sqrt{-t}D_t(A_t,C_t^y))= \cosh(\sqrt{-t}D_t(C_t^y,B_t^x)).
\end{eqnarray*}
\end{proposition}  
 
 \section{Transition of angles and of area}\label{s:angle}
Using the Pythagorean theorem in spherical and hyperbolic geometry, it is possible to prove the other formulae in the above table, in particular the formulae
\[\cos A_t=\frac{\tan \sqrt{a_t}}{\tan \sqrt{b_t}}\]
in spherical geometry
and
\[\cos A_t=\frac{\tanh \sqrt{a_t}}{\tanh \sqrt{b_t}}\]
in hyperbolic geometry.

At the same time, we can study convergence of angles. For this take a triangle $ABC$ with right angle at $C$ and let $\alpha$ be the angle at $A$. For each $t$, the (lengths of the) sides opposite to $A,B,C$ are respectively $a_t,b_t,c_t$ in the geometry $E_t$. Take the side lengths $a$ and $b$ to be constant, $a_t=a$ and $b_t=b$.
For each $t$, we have, in the geometry $E_t$,
\[\cos \alpha_t = \frac{\tan \sqrt{t} a}{\tan  \sqrt{t} b}.\]
As $t\to 0$, we have $\cos \alpha_t \to a/b$. Thus, the angle $\alpha_t$ in the triangle $A_tB_tC_t$  converges indeed to the Euclidean angle. 

Note that this can also be used to define the notion of angle in each geometry. The reader can check that this definition amounts to the one we gave in \S \ref{s:trans}.

Finally we consider transition of area.

There are several works dealing with non-Euclidean area. For instance, in the paper \cite{Euler-Variae}, Euler obtained the following formula for the area $\Delta$ of a spherical triangle of side lengths $a,b,c$:
\begin{equation}\label{f:Euler}
\cos \frac{1}{2}\Delta= \frac{1+\cos a+\cos b+\cos c}{4\cos\frac{1}{2}a\cos\frac{1}{2}b\cos\frac{1}{2}c }.
\end{equation}
This formula should be compared to the Heron formula in Euclidean geometry that gives the area $\Delta$ of a triangle in terms of its side length. In the geometry $E_t$ ($t>0$), Euler's formula becomes
\[
 \cos \frac{1}{2}\Delta= \frac{1+\cos \sqrt{t}a+\cos  \sqrt{t}b+\cos  \sqrt{t} c}{4\cos\frac{ \sqrt{t}}{2}a\cos\frac{ \sqrt{t}}{2}b\cos\frac{ \sqrt{t}}{2}c }.
\]
 Setting $s=\sqrt{t}$ to simplify notation, we write this formula as
 \begin{equation} \label{Euler1} 
  \cos \frac{1}{2}\Delta= \frac{1+\cos sa+\cos  sb+\cos  s c}{4\cos\frac{ s}{2}a\cos\frac{ s}{2}b\cos\frac{ s}{2}c }.
  \end{equation}
We would like to see that Euler's formula (\ref{Euler1}) transits and leads to the area formula in Euclidean geometry for $t=0$. 
It suffices to deal with right triangles, and we therefore assume that the angle at $c$ is right. We use the Pythagorean theorem 
\[\cos  \sqrt{t}c = \cos \sqrt{t}a\cos \sqrt{t}b.\] 
 
We take the Taylor expansion in $t$ in Formula (\ref{Euler1}).

 Using the formula for the cosine of the double and taking the square, we obtain the formula \[
(\cos \frac{1}{2}\Delta)^2=\frac{(1+\cos sa + \cos sb + \cos sa \times \cos sb)^2}{2(1+\cos sa)(1+\cos sb)(1+\cos sa\times\cos sb).}
.\]
The degree-8 Taylor expansion of $1- (\cos \frac{1}{2}\Delta)^2$ is 
\[\frac{1}{16} a^2b^2s^4+ (\frac{1}{96} a^2b^4+\frac{1}{96}a^4b^2)+O(s^8).\]
The first term in this expansion is $\frac{1}{16} a^2b^2s^4$, that is, $\frac{1}{16} a^2b^2t^2$, which is the square of the expression $\frac{1}{4} abt$. Recall now that for the transition of the distance function (\S \ref{s:trans}), we had to normalize the length function in $E_t$, dividing it by $\sqrt{t}$. It is natural then, for the transition of the area function, to divide it by $t$. Thus, the result that we obtain is $\frac{1}{2} ab$, which is indeed the area of a Euclidean triangle with base $a$ and altitude $b$.

%
%

Finally, we note that there has been a recent activity on transition of geometries in dimension $3$, namely, on moving continuously between the eight Thurston geometries, and also on varying continuously between Riemannian and Lorentzian geometries on orbifolds. We mention the works of Porti \cite{Porti2002}, \cite{Porti2010} and Porti and Weiss \cite{Porti-Weiss} and Cooper, Hodgson and Kerckhoff \cite{CHK},  Kerckhoff and Storm \cite{Kerckhoff-Storm} and Dancieger \cite{Dancieger}.  The ideas and the methods are different from those of the present paper.

As a conclusion, we propose the following problem.
\begin{problem}Extend  this theory of transition of geometries to a  parameter space which, instead of being the interval $[-1,1]$, is a neighborhood of the origin in the complex plane $\mathbb{C}$.
\end{problem}

This is in accordance with a well-established tradition to trying to extend to teh complex world. We recall in this respect the following words of Painlev\'e \cite{Painleve} ``Entre deux v\'erit\'es du domaine r\'eel, le chemin le plus facile et le plus court passe bien souvent par le domaine complexe." (Between two truths of the real domain, the easiest and shortest path quite often passes through the complex domain.) \footnote{Painelv\'e's sentence has been taken taken over by Hadamard in \cite{Hadamard}: ``It has been written that the shortest and best way between two truths of the real domain often passes through the imaginary one."} We also recall Riemann's words, from his \emph{Inaugural dissertation} \cite{Riemann-Grundlagen}, concerning the introduction of complex numbers:  ``If one applies these laws of dependence in an extended context, by giving the related variables complex values, there emerges a regulatrity and harmony which would otherwise have remained concealed."

\end{document}